\documentclass[a4paper,twoside,final,reqno]{amsart}
\usepackage[utf8]{inputenc}
\usepackage[T1]{fontenc}
\usepackage[english]{babel}
\usepackage{etex,fixltx2e}
\usepackage{lmodern}
\usepackage{amsmath,amssymb,ifthen}
\usepackage{mathtools}
\usepackage{hyperref}
\usepackage{tikz}

\newcommand{\bbA}{{\mathbb A}}
\newcommand{\bbC}{{\mathbb C}}
\newcommand{\bbN}{{\mathbb N}}
\newcommand{\bbQ}{{\mathbb Q}}
\newcommand{\bbR}{{\mathbb R}}
\newcommand{\bbZ}{{\mathbb Z}}
\newcommand{\calK}{{\mathcal K}}
\newcommand{\calL}{{\mathcal L}}
\newcommand{\calP}{{\mathcal P}}
\newcommand{\calO}{{\mathcal O}}

\newcommand{\frakI}{{\mathfrak I}}

\newcommand{\frakM}{{\mathfrak M}}
\newcommand{\frakR}{{\mathfrak R}}

\newcommand{\Ima}{{\ensuremath{\operatorname{\frakI}}}}
\newcommand{\Rea}{{\ensuremath{\operatorname{\frakR}}}}

\newcommand{\conv}{{\ensuremath{\operatorname{conv}}}}
\newcommand{\lin}{{\ensuremath{\operatorname{lin}}}}

\newcommand{\rank}{{\ensuremath{\operatorname{rank}}}}
\newcommand{\vol}{{\ensuremath{\operatorname{vol}}}}
\newcommand{\isom}{\mathrel{\cong}}
\newcommand{\GL}{{\ensuremath{\operatorname{GL}}}}
\DeclarePairedDelimiter\abs{\lvert}{\rvert}

\newcommand{\fixedv}{{\overline{v}}}

\DeclarePairedDelimiterX\Sett[2]{\{}{\}}{\,#1\vphantom{R^m}\,\delimsize\vert\,#2\vphantom{R^m}\,}
\DeclarePairedDelimiterX\Set[1]{\{}{\}}{\,#1\vphantom{R^m}\,}
\DeclarePairedDelimiterX\grind[2]{[}{]}{\,#1\vphantom{R^m_i}\,\mathord{:}\,#2\vphantom{R^m_i}\,}

\theoremstyle{plain}
\newtheorem{thm}{Theorem}[section]
\newtheorem{prop}[thm]{Proposition}
\newtheorem{lem}[thm]{Lemma}

\theoremstyle{definition}
\newtheorem{rem}[thm]{Remark}

\newtheorem{defn}[thm]{Definition}
\newtheorem{exmp}[thm]{Example}

\usepackage[final]{pdfcomment}

\numberwithin{equation}{section}
\begin{document}

\title[Note on adelic triangulations]{Note on adelic triangulations and an Adelic Blichfeldt-type inequality}
\author{Martin Henk}
\address{Fakult\"at f\"ur Mathematik, Otto-von-Guericke-Universit\"at
  Magdeburg, Universit\"atsplatz 2, 39106 Magdeburg, Germany}
\email{martin.henk@ovgu.de, carsten.thiel@gmail.com}
\author{Carsten Thiel}



\begin{abstract}
  We introduce a notion of convex hull and polytope into adele space.
  This allows to consider adelic triangulations which, in particular, 
  lead to an adelic blichfeldt-type inequality, complementing former results.
\end{abstract}


\maketitle

\newcommand{\idealclassbound}{{\ensuremath\Gamma_K}}

\section{Introduction}
Let $\calK^m_0$ be the set of all $0$-symmetric convex bodies in the
$m$-dimensional Euclidean space $\bbR^m$ with non-empty interior, i.e., 
$C\in \calK^m_0$ is an $m$-dimensional  compact convex set satisfying
$C=-C$.  
An important 
subclass of convex bodies are formed by
polytopes $P=\conv\{v_1,\dots,v_l \}$, i.e.,
the convex hull of finitely many points $v_1,\ldots,v_l\in\bbR^m$. We
write $\calP^m$ and  $\calP_0^m$ for the set of all, respectively the
set of all $0$-symmetric, $m$-dimensional polytopes in $\bbR^m$.

By a lattice $\Lambda\subset \bbR^m$ we understand a free
$\bbZ$-module of rank $\rank\Lambda \leq m$.   The set of all lattices
in $\bbR^m$ is denoted by $\calL^m$, and $\det\Lambda$ denotes the determinant of $\Lambda\in\calL^n$,
that is the $(\rank\Lambda)$-dimensional volume of a fundamental cell
of $\Lambda$. For more detailed information on lattices we refer to
\cite{Gruber:2007um, Gruber:1987vp}.   

One of the classical inequalities relating a convex body  and
 points of a lattice is Blichfeldt's inequality  
from 1921 (see, e.g.,~\cite{Blichfeldt:1920vv}). It gives an upper bound on the
number of lattice points of a lattice $\Lambda\in \calL^m$ contained
in a convex body $C\in\calK^m$ under the assumption that
$\dim_\bbR(C\cap\Lambda)=m$, i.e., $C\cap\Lambda$ contains $m+1$
affinely independent lattice points 
\begin{equation}\label{eq:classicalblichfeldt}
  \abs{C\cap\Lambda} \leq m!\, \frac{\vol(C)}{\det\Lambda} +m. 
\end{equation} 
The bound is sharp, for instance for $\Lambda=\bbZ^m$ and simplices 
of the form $\conv\Set[\big]{0,\ell e_1,\allowbreak e_2,\ldots,e_m}$,
where $\ell\in\bbN$ and $e_1,\ldots,e_m$ are the standard unit
vectors.  The additional requirement on the dimension is necessary, as an axis-parallel box with 
very small edge length in one direction can contain a large number of lattice points,
while still having arbitrarily small volume.

The usually way to prove Blichfeldt's result, as many other results in the
context of the interplay of lattice points and convex bodies, is via
triangulations (cf., e.g., \cite{Beck:2007vv, Loera:2010p9983}).  To this end one firstly notices that  by replacing $C$
by $\conv\{C\cap \Lambda\}$ 
it suffices to
prove the bound for the class of lattice polytopes $P\in \calP^m$,
i.e., polytopes admitting a representation as
$\conv\{w_1,\dots,w_l\}$ where $w_i$ are points of a lattice $\Lambda$.
The next observation is that $P$ can be triangulated  in at least
$|P\cap\Lambda|-m$  many lattice simplices and the volume of a lattice
simplex is at least $\det\Lambda/m!$.

In this note we want to introduce the notion of convex hull and
triangulations in the adele space, which has been proved in recent
years as an
 excellent and challenging space for extensions  and generalizations
 of classical concepts from Geometry of numbers, see e.g.,
 \cite{Bombieri:2009bo, Bombieri:1983uz,
   Fukshansky:2006ti, Fukshansky:2006db, Fukshansky:2010iw, Fukshansky:2013vo,  Gaudron:2008df, Gaudron:2009iu,
   Gaudron:2012hq,   
   Roy:1995p7975, Thiel:2012iz, Thunder:1998p7981, Thunder2002}
    as the references within.  

After a short introduction to adelic geoemtry in Section~\ref{sec:adelicgeometry}, we
introduce in Section~\ref{sec:adelicpolytopes} our notion of the adelic convex hull and 
adelic polytopes. In particular, we will prove a lower bound on the
adelic volume of an adelic lattice simplex in the case of totally real
fields (Lemma \ref{lem:adelicsimplexvolume}). 
In Section~\ref{sec:adelictriangulation} we  study 
adelic triangulations, which we use to prove our main result  
\begin{thm}\label{thm:adelicblichfeldttotallyreal}
  Let $K$ be a totally real number field of degree $d=[K:\bbQ]$.
  Let $C$ be an adelic convex body with $\dim_K(C\cap K^n)=n$. 
  Then
  \[
    \abs[\big]{C\cap K^n} \leq (n!)^d \vol_\bbA(C) +n.
  \]
\end{thm}
The necessary  notations will be introduced in Section \ref{sec:adelicgeometry}. 
We remark that for $d=1$ and  $n=m$ we get Blichfeldt's inequality
\eqref{eq:classicalblichfeldt}. Moreover, 
Theorem~\ref{thm:adelicblichfeldttotallyreal} improves for the special
case of totally real fields on a former
more general result of Gaudron \cite{Gaudron:2009iu} (see \eqref{eq:gaudronsblichfeldt}) for
$0$-symmetric adelic convex bodies.  This result, as well as other
adelic symmetric variants of Blichfeldt's
theorem will be presented in the final  Section~\ref{sec:symmcase}.

\section{Adelic geometry}\label{sec:adelicgeometry}
In this section we will briefly introduce the notations and concepts
from adelic geometry used in the following sections. For a detailed
discussion we refer to \cite{Bombieri:2009bo, Lang:1994cr, Neukirch:1999ky}.

Let $K$ be an algebraic number field of degree $d=[K:\bbQ]$.
Let $r$ be the number of real and $s$ the number of pairs of complex embeddings 
of $K$ into $\bbC$, so $d=r+2s$. Denote by 
$\calO$ the ring of algebraic integers of $K$ and by $\Delta_K$ and
$h_K$ its field discriminant and class number, respectively. 

Let $M(K)$ be the set of all places of $K$.
For $v\in M(K)$ we write $v\nmid\infty$ for non-archimedean places
and  $v\mid\infty$ for the archimedean ones.
We write $\abs{\,\cdot\,}_v$ for the corresponding absolute value on $K$.
We normalise it to extend either the usual absolute value on $\bbQ$ for archimedean places 
or the $p$-adic absolute value for a prime $p$.
Then the local field $K_v$ is the completion of $K$ with respect to $v$.
For $v\nmid\infty$ let $\calO_v$ be the local ring of integers of $K_v$.

Let $K_\bbA$ be the ring of adeles of $K$ and $K_\bbA^n$ the standard module of rank $n\geq 2$,
i.e.,\ the $n$-fold product of adeles.
Recall that $K_\bbA$ is the restricted direct product of the $K_v$ with respect to the $\calO_v$.
For any $v\in M(K)$ let $d_v=[K_v:\bbQ_v]$ be the local degree ($\bbQ_\infty\isom\bbR$).
Then 
\begin{equation}\label{eq:productformulaetc}
 d=\sum_{v\mid \infty} d_v\,,\quad\text{and for all non-zero $a\in K$}\quad
\prod_{v\in M(K)} \abs{a}_v^{d_v}=1.
\end{equation}
For $v\nmid\infty$ let $\mu_v$ be the Haar measure on $K_v$ normalized such that
$\mu_v(\calO_v)=1$.
Thus for any ideal $\alpha\calO_v\subseteq\calO_v$ we get
$\mu_v(\alpha\calO_v)=\abs{\alpha}_v^{d_v}$.
For $v\mid\infty$ let $\mu_v$ be the Lebesque measure on $K_v=\bbR$ 
resp.\ twice the Lebesque measure on $K_v=\bbC$. 
Define the Haar measure $\vol_\bbA$ on $K_\bbA^1$ by
\[\vol_\bbA=\frac{1}{\sqrt{\abs{\Delta_K}}} \ \prod_{v\in M(K)} \mu_v\]
and use the product measure on $K_\bbA^n$.

\begin{defn}[Adelic convex body]\label{def:adelicconvexbody}
For each $v\nmid\infty$ let $C_v$ be a free $\calO_v$-module of full rank,
where $C_v=\calO_v^n$ for all but finitely many $v$.
In other words, for any $v\nmid\infty$ there is an $A_v\in\GL_n(K_v)$ such that
$C_v=A_v\calO_v^n$, where $A_v\in\GL_n(\calO_v)$ for all but finitely many $v$.
For $v\mid\infty$ we have $K_v\isom\bbR$ or $K_v\isom\bbC$. 
In this case let $C_v\subset K_v^n$ be a compact convex body with 
non-empty interior in $\bbR^n$ or $\bbC^n\isom\bbR^{2n}$ respectively, i.e.,
$C_v\in\calK^n$ or $C_v\in\calK^{2n}$.
Then the set
\[C = \prod_{v\nmid\infty} C_v \times \prod_{v\mid\infty} C_v\]
is called an \emph{adelic convex body}.
If $C_v$ is symmetric for $v\mid\infty$, i.e., $C_v\in\calK_0^n$ or $C_v\in\calK_0^{2n}$,
we call $C$ a \emph{$0$-symmetric adelic convex body}.
\end{defn}

For $(x_v)_v\in K_\bbA^n$ we define the scalar multiple $(y_v)_v=\lambda(x_v)_v$ for $\lambda\in\bbR^+$ by
\[
y_v\coloneqq\begin{cases}\phantom{\lambda}x_v &\text{if } v\nmid\infty\,,\\
\lambda\, x_v &\text{if } v\mid\infty\,.\end{cases}
\]

Denote by $\sigma_i$, $1\leq i \leq r$, the embeddings of $K$ into $\bbR$
and by $\sigma_{r+i}=\overline{\sigma}_{r+i+s}$, $1\leq i\leq s$, the pairs 
of embeddings of $K$ into $\bbC$. If $s=0$, 
 then $K$ is called a \emph{totally real} field. For instance,
 $\bbQ[\sqrt{2}]$ is  totally real, but not $\bbQ[\sqrt[3]{2}]$.

Let $\overline{\,\cdot\,}$ denote  complex conjugation in $\bbC$,
cf.~\cite{Blanksby:1978tw}. Then
\begin{align*}
  \iota &\colon x\mapsto\bigl(\sigma_1(x),\ldots,\sigma_{r}(x),
  \sigma_{r+1}(x),\ldots,\sigma_{r+s}(x)\bigr)\\
\shortintertext{and}
\overline{\iota} &\colon x\mapsto\bigl(\sigma_1(x),\ldots,\sigma_{r}(x),
  \overline{\sigma}_{r+1}(x),\ldots,\overline{\sigma}_{r+s}(x)\bigr)
\end{align*}
are embeddings of $K$ into $K_\infty=\prod_{v\mid\infty}K_v$. 
There is a canonical isomorphism $\rho\colon K_\infty\rightarrow \bbR^{d}$ with
\begin{equation}\begin{multlined}\label{eq:isombetweenKinftyandRnd}
  \rho\bigl(x_1,\ldots,x_{r},x_{r+1},\ldots,x_{r+s}\bigr)= \\
  \qquad\bigl(x_1,\ldots,x_{r},\Rea(x_{r+1}),\Ima(x_{r+1}),\ldots,\Rea(x_{r+s}),\Ima(x_{r+s})\bigr)\,.
\end{multlined}\end{equation}
Here $\Rea$ and $\Ima$ denote real and imaginary parts, respectively.

Together we get a map $(\rho\circ\iota)\colon K\hookrightarrow \bbR^d$,
that sends a field element to the vector whose entries are the images under the real and complex embeddings,
splitting the latter points into real and imaginary part,
\[
  x\mapsto\bigl(\sigma_1(x),\ldots,\sigma_{r}(x),
      \Rea(\sigma_{r+1}(x)),\Ima(\sigma_{r+1}(x)),\ldots,\Rea(\sigma_{r+s}(x)),\Ima(\sigma_{r+s}(x))\bigr)\,.
\]
In the rank-$n$-case let $K_\infty^n=\prod_{v\mid\infty}K_v^n$,
\[
  \iota^n \coloneqq (\sigma_1^n,\ldots,\sigma_{r}^n,\sigma_{r+1}^n,\ldots,\sigma_{r+s}^n) 
  \colon K^n\rightarrow K_\infty^n\,,
\]
and $\overline{\iota}^n$ analogously,
where the $\sigma_i$ act component-wise.
Similarly $\rho^n\colon K_\infty^n\rightarrow\bbR^{nd}$.
To simplify notation, we usually write $\rho$ and $\iota$ instead of $\rho^n$ and $\iota^n$.

Throughout the paper, we use the following notation 
\begin{equation}\label{eq:def:frakm}
  C_\infty=\prod_{v\mid\infty} C_v
  \quad\text{and}\quad
  \frakM=\bigcap_{v\nmid\infty}\bigl(C_v\cap K^n\bigl)\,.
\end{equation}
Observe that our standard embedding $\rho\circ\iota : K^n\hookrightarrow\bbR^{nd}$,
cf.~\eqref{eq:isombetweenKinftyandRnd} and thereafter, is injective, 
and therefore
\begin{equation}\label{eq:adelicandclassciallatticepointsarethesame}
  \abs[\big]{C \cap K^n}= \abs[\Big]{\rho\Bigl( \prod_{v\mid\infty} C_v \Bigr) \cap \rho(\iota(\frakM)) }\,.
\end{equation}
We will use this important connection for some of the proofs below.

\section{adelic polytopes}\label{sec:adelicpolytopes}

We start by giving local definitions of \emph{convex hull} of points 
$\overline{a}_0,\ldots,\overline{a}_m\in K_\bbA^n$, where
$\overline{a}_{k,v}$ is the $v$-entry of $\overline{a}_{k}$ for  $v\in
M(K)$. 
To exclude degenerate cases, we always require that for all $v\in M(K)$ we have
\[
  \lin_{K_{v}}\Sett*{\overline{a}_{k,v} }{ 0\leq k\leq m}=K_v^n\,,
\]
and for all but finitely many $v\nmid\infty$, the entries of
$\overline{a}_{k,v}$ are in $\calO^\ast_v$, where, as usual,
$\calO^\ast_v$ denotes the group of units in $\calO_v$.

For $v\nmid\infty$ define the module 
\begin{equation}\label{eq:localpidgenerator}
  C_v =\conv_v\Set*{\overline{a}_{0,v},\ldots,\overline{a}_{m,v}}=
  \calO_v\overline{a}_{0,v}+\ldots+\calO_v\overline{a}_{m,v}\,.
\end{equation}
Note that $C_v$ is an $\calO_v$-module in $K_v^n$ of full rank.
In fact, we have $C_v=\lin_{\calO_v}\{\overline{a}_{0,v},\allowbreak\ldots,\overline{a}_{m,v}\}$,
the minimal $\calO_v$-module in $K_v^n$ containing all the points. 
Since the $K_v$ are local fields, there exist $A_v\in\GL_n(K_v)$, such that
$C_v=A_v\calO^n_v$. Note that in general for $t\in K^n$
\[
  \conv_v\Set*{\overline{a}_{0,v}+t,\ldots,\overline{a}_{m,v}+t}\neq\conv_v\Set*{\overline{a}_{0,v},\ldots,\overline{a}_{m,v}}+t,
\]
but if $t\in\Set*{-\overline{a}_{0,v},\ldots,-\overline{a}_{m,v}}$, we
certainly have 
\begin{equation}\label{eq:vfiniteconvinclusion}
  \conv_v\Set*{\overline{a}_{0,v}+t,\ldots,\overline{a}_{m,v}+t}\subseteq\conv_v\Set*{\overline{a}_{0,v},\ldots,\overline{a}_{m,v}},
\end{equation}
as the points on the left are contained in the $\bbZ$-span of the points on the right.

For $v\mid\infty$ real, let
\begin{align}
  C_v&=\conv_\bbR\Set*{\overline{a}_{0,v},\ldots,\overline{a}_{m,v}}\label{eq:defrealnonsymmconvhull}\\
  &=\Sett[\Bigg]{\sum_{i=0}^m \lambda_i \overline{a}_{i,v}}{\lambda_i\in\bbR\,,\ 0\leq\lambda_i\leq 1\,,\ \sum_{i=0}^m \lambda_i=1}\subset\bbR^n\nonumber
\end{align}
and
\begin{align}
  C_v^\diamond&=\conv_\bbR\Set*{\pm\overline{a}_{0,v},\ldots,\pm\overline{a}_{m,v}}\label{eq:defrealsymmconvhull}\\
  &=\Sett[\Bigg]{\sum_{i=0}^m \lambda_i \overline{a}_{i,v}}{\lambda_i\in\bbR\,,\ 0\leq \abs{\lambda_i}\leq 1\,,\ \sum_{i=0}^m
  \abs{\lambda_i}\leq 1}\subset\bbR^n\,.\nonumber
\end{align}
These are the standard convex hull of points in real space and its symmetric
variant. They are equivalent to defining the bodies as the intersection of all (symmetric)
 convex bodies containing the points $\overline{a}_{0,v},\ldots,\overline{a}_{m,v}$.

For $v\mid\infty$ complex, we only define the symmetric body
\begin{align}
  C_v^\diamond&=\conv_\bbC\Set*{\overline{a}_{0,v},\ldots,\overline{a}_{m,v}}\label{eq:defcomplexsymmconvhull}\\
  &=\Sett[\Bigg]{\sum_{i=0}^m \lambda_i \overline{a}_{i,v}}{\lambda_i\in\bbC\,,\ 0\leq
    \abs{\lambda_i}\leq 1\,,\ \sum_{i=0}^m \abs{\lambda_i}\leq 1}\subset\bbC^n\,.\nonumber
\end{align}
By construction, this is the intersection of all symmetric convex
bodies in complex space containing
the points $\overline{a}_{0,v},\ldots,\overline{a}_{m,v}$.
We are not aware of any more general notion of convex hull in complex spaces.
When identifying $\bbC^n\isom\bbR^{2n}$, we can use the definitions used in the real case,
but the bodies obtained in this way lie in a real (affine) subspace of $\bbR$-dimension $n$ in $\bbR^{2n}$
and do thus not define an adelic convex body. This is the reason why
we will consider arbitrary  adelic polytopes only in the case of
totally real fields.   

Using our constructions of convex hull, we can now define the following special classes of adelic convex bodies.

\begin{defn}[Adelic convex hull, polytopes, simplices and cross-polytopes]
  Given $\overline{a}_0,\ldots,\overline{a}_m\in K_\bbA^n$ as before, we define
  \[
    C^\diamond=\conv_\bbA^\diamond\Set*{\overline{a}_0,\ldots,\overline{a}_m}
    = \prod_{v\nmid\infty} A_v \calO_v^n \times\prod_{v\mid\infty} C_v^\diamond
  \]
  with $A_v$ implicitly defined by \eqref{eq:localpidgenerator} above
  and $C_v^\diamond$ as in \eqref{eq:defrealsymmconvhull} and \eqref{eq:defcomplexsymmconvhull},
  as the \emph{symmetric adelic convex hull} of $\overline{a}_0,\ldots,\overline{a}_m$.
  \index{adelic!convex hull}\index{convex hull}
  If $K$ is totally real, we define the \emph{adelic convex hull} of
  $\overline{a}_0,\ldots,\overline{a}_m$ as
  \[
    C=\conv_\bbA\Set*{\overline{a}_0,\ldots,\overline{a}_m}
    = \prod_{v\nmid\infty} A_v \calO_v^n \times\prod_{v\mid\infty} C_v\,,
  \]
  where $C_v$ is defined as in \eqref{eq:defrealnonsymmconvhull}.
  In case $m=n$, we speak of the \emph{adelic cross-polytope} and \emph{adelic simplex} respectively.
  All of these bodies will also be called \emph{adelic polytopes}
  and if $C$ is an adelic polytope, it can be written as $C=\prod_{v\in M(K)}C_v$ with local
  bodies $C_v$ defined as above.
\end{defn}

Denote by $\sigma_v\colon K\rightarrow K_v$ the inclusion of $K$
into $K_v$ and by abuse of notation
also $\sigma_v\colon K^n\rightarrow K_v^n$ for all places
$v$ of $K$.

\begin{defn}[Adelic lattice polytopes, simplices, cross-polytopes]
  Given points $a_0,\ldots,a_m\in K^n$ that span $K^n$, 
  identify 
  \[
    a_k=\overline{a}_k=\left(\sigma_v(a_k) \middle| v\in M(K)\vphantom{R^m}\right)\,.
  \]
  Then
  \[
    C=\conv_\bbA\Set*{a_0,\ldots,a_m}
      \quad\text{and}\quad
    C^\diamond=\conv_\bbA^\diamond\Set*{a_0,\ldots,a_m}
  \]
  are the \emph{adelic lattice polytope} and \emph{symmetric adelic lattice polytope} 
  generated by  $a_0,\ldots,a_m$, respectively.
  The body $C$ is, of course, again only defined for $K$ totally real,
  and for $m=n$ we call $C$ an \emph{adelic lattice simplex}. 
If   additionally $a_0=0$, $C^\diamond$ will be called an \emph{adelic lattice cross-polytope}.
\end{defn}

\begin{rem}
  The intersection of two adelic polytopes is again an adelic polytope,
  since this property holds for all $v$ and two adelic polytopes differ only for
  finitely many $v$ and for arbitrary sets $X$, $Y$ and $Z$ we have
  \[
    (X\times Y)\cap(X\times Z)=\Sett*{(x,y)}{x\in X\,,\ y\in Y\,,\ y\in Z}
    =X\times(Y\cap Z)\,.
  \]

  The intersection of two adelic lattice polytopes, however, is not an adelic lattice
  polytope in general
  .
\end{rem}

Adelic polytopes are not as nice as their classical real counterparts.
In Euclidean space, a polytope $P\in\calP^m$ can be written as both
the convex hull of a finite number of points or as the intersection of 
a finite number of closed half-spaces.
Given a linear functional $\ell:\bbR^m\rightarrow \bbR$, the kernel of $\ell$ is a hyperplane $H$
and using the ordering of $\bbR$, we decide whether two points $x_1,x_2\in\bbR^m$ lie on the same side of $H$
by comparing the signs of $\ell(x_1)$ and $\ell(x_2)$ and thereby also defining two half-spaces.

Such a construction is not possible in the adelic setting, as we do not have an ordering on $K_\bbA$.

Unfortunately, we also cannot expect to have an adelic counterpart to
the classical Ehrhart-theory for lattice polytopes. To this end we
recall that  
for a lattice polytope $P=\conv\Set{v_1,\ldots,v_s}\in\calP^m$, where 
$v_i$ are lattice points of a rank $m$ lattice $\Lambda$, we know by a theorem of
Ehrhart \cite{Ehrhart1962} that 
the number of lattice points in $k\, P$ for a positive integer $k$
is given by a polynomial of degree $m$, the Ehrhart polynomial\index{Ehrhart polynomial}
\[
  \abs[\big]{k P\cap \Lambda} = \sum_{i=1}^m G_i(P,\Lambda) k^i\,,\quad k\in\bbN\,.
\]
The polynomial is unique and the coefficients depend only on $P$ and $\Lambda$.
The behaviour and properties of this polynomial 
have been studied intensively,
see Beck and Robins~\cite{Beck:2007vv} for an overview as well as
e.g.\ McMullen~\cite{McMullen:1978p10992} and 
Linke~\cite{Linke:2011bw} for more specific results.

Now consider an adelic lattice polytope $C=\prod_{v} C_v$, and 
let $C_\infty$ and $\frakM$ be as before, cf.~\eqref{eq:def:frakm}.
Then $\rho(C_\infty)$ is a polytope in $\bbR^{nd}$, 
as the factors of $C_\infty$ are polytopes.
On account of \eqref{eq:adelicandclassciallatticepointsarethesame},
the number $\abs[\big]{k C\cap K^n}$ has to grow like $k^{nd}$.
On the other hand, the body $\rho(C_\infty)$ is in general not a lattice polytope
with respect to the lattice $\rho(\iota(\frakM))$.

\begin{exmp}
Consider for example $K=\bbQ[\sqrt{2}]$ for $n=1$ and the body $C=\prod_{v\nmid\infty}\calO_v\times[-1,1]^2$, 
which is an adelic lattice polytope in the sense above, as $C=\conv_\bbA\Set{\pm 1}$.
Observe, that $\frakM=\calO$.
Figure~\ref{fig:nubbqsqrt2} shows the embedding of $C$ into real space $\bbR^{nd}$.
\begin{figure}[ht]
\centering
\begin{tikzpicture}[scale=0.9]
\clip(-4.7,-3.1) rectangle (4.7,3.1);
\draw[gray!50,dashed,thick] (0,-4.5) -- (0,4.5);
\draw[gray!50,dashed,thick] (-4.5,0) -- (4.5,0);
\foreach \i in {-4,...,4}
{
  \draw (\i,\i) node[fill,shape=circle,scale=.3] {};
}
\foreach \i in {-3,...,3}
{
  \draw (\i-1.41,\i+1.41) node[fill,shape=circle,scale=.3] {};
  \draw (\i+1.41,\i-1.41) node[fill,shape=circle,scale=.3] {};
}
\foreach \i in {-1,...,1}
{
  \draw (\i-2.82,\i+2.82) node[fill,shape=circle,scale=.3] {};
  \draw (\i+2.82,\i-2.82) node[fill,shape=circle,scale=.3] {};
}
\foreach \i in {0,...,0}
{
  \draw (\i-4.23,\i+4.23) node[fill,shape=circle,scale=.3] {};
  \draw (\i+4.23,\i-4.23) node[fill,shape=circle,scale=.3] {};
}
\draw[thick] (1,1) -- (-1,1) -- (-1,-1) -- (1,-1) -- cycle;
\draw (-0.05,-0.1) node[anchor=south west]  {\scriptsize$\rho(0)$};
\draw (1,1.2) node[anchor=north west]  {\scriptsize$\rho(1)$};
\draw (1.45,-1.7) node[anchor=south west]  {\scriptsize$\rho(\sqrt{2})$};
\draw (2.7,-.41) node[fill=white,anchor=west]  {\scriptsize$\rho(1+\sqrt{2})$};
\draw (-1.1,-0.5) node[fill=white,anchor=east] {\scriptsize$\rho(C_\infty)$};
\end{tikzpicture}
\caption{Embedding of an adelic convex body into real space.}
\label{fig:nubbqsqrt2}
\end{figure}
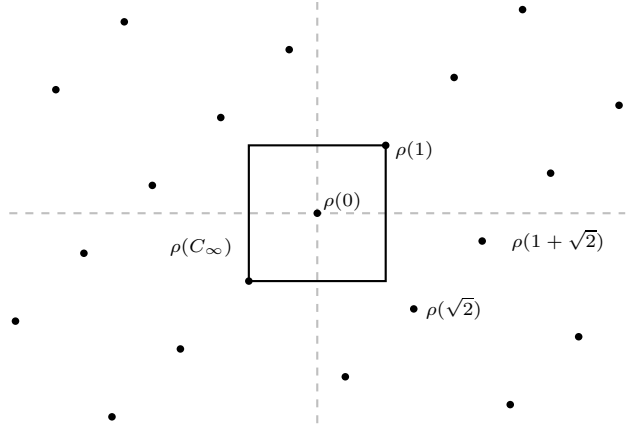
It is evident from the figure, that of the four vertices of $\rho\bigl([-1,1]^2\bigr)$ 
only $(1,1)$ and $(-1,-1)$ are lattice points but not $(1,-1)$ and $(-1,1)$
and thus $\rho(C_\infty)$ is not a lattice polytope with respect to $\rho(\iota(\calO))$.
However, the infinite part of any adelic convex body has to be of the form $C_\infty=[-a,a]\times[-b,b]$ for $a,b\in\bbR$.
But since $\rho(\iota(\calO))$ is generated by $(1,1),(-\sqrt{2},\sqrt{2})\in\bbR^2$,
no box can have only lattice points as vertices. Thus 
the image of $C$ under the embedding into $\bbR^2$ can not be a
lattice polytope, and 
we can therefore not find an adelic analogue to the Ehrhart
polynomial.
\end{exmp}

Next we deal with the adelic volume of lattice simplices. 
Let $e_1,\ldots,e_n$ be any basis of $K^n$, then it is known 
that the volume of the adelic lattice cross-polytope 
$C^\diamond=\conv_\bbA^\diamond\Set{e_1,\ldots,e_n}$
is (cf, e.g., \cite{Bombieri:1983uz})
\begin{equation}\label{lem:volumeofadeliccrosspolytope}
  \vol_\bbA(C^\diamond)=\frac{2^{dn}\pi^{s n}}{(n!)^{r} ((2n)!)^{s}}.
\end{equation} 
If $K$ is totally real and
$S=\conv_\bbA\Set{0,e_1,\ldots,e_n}=\prod_{v\in M(K)}S_v$ 
is an adelic lattice simplex, then for the local simplices $S_v$ 
at the infinite places we have  
\begin{equation}\label{eq:volumeofeuclideansimplex}
  \vol_v(S_v)=\frac{1}{n!}\qquad\text{for $v\mid\infty$ real.}
\end{equation}

\begin{lem}\label{lem:adelicsimplexvolume}
  Let $K$ be a totally real number field of degree $d=[K:\bbQ]$ and
  let $S=\conv_\bbA\Set{a_0,\ldots,a_n}$ be an adelic lattice simplex. Then
  \[
    \vol_\bbA(S)\geq \frac{1}{(n!)^d}\,.
  \]
\end{lem}

\begin{proof}
  On account of \eqref{eq:vfiniteconvinclusion}, we may assume
  w.l.o.g.\ $a_0=0$, possibly switching to a subset at some finite
  places. Let $A=(a_1\ldots a_n)$ be the matrix whose columns are
  $a_1,\ldots,a_n$, and let $S=\prod_{v\in M(K)}C_v$.

  For $v\mid\infty$ we get from \eqref{eq:volumeofeuclideansimplex} that
  \begin{align*}
   \vol_v(C_v) &= \vol_v\bigl(\conv_\bbR\Set*{\sigma_v(a_1,),\ldots,\sigma_v(a_n)}\bigr)\\
    &=\frac{\abs*{\det\bigl(\sigma_v(a_1)\ldots\sigma_v(a_n)\bigr)}_\infty}{n!}
    =\frac{\abs*{\sigma_v\bigl(\det(A)\bigr)}_\infty}{n!}=\frac{\abs*{\sigma_v\bigl(\det(A)\bigr)}_\infty}{n!}\,.
  \end{align*}

  On the other hand, for $v\nmid\infty$, we have
  \[
    C_v=\calO_v \sigma_v(a_1) +\ldots+\calO_v \sigma_v(a_n)
       =\bigl(\sigma_v(a_1)\ldots\sigma_v(a_n)\bigr)\calO_v^n
  \]
  and thus
  \[
    \vol_v(C_v)=\abs*{\det\bigl(\sigma_v(a_1)\ldots\sigma_v(a_n)\bigr)}^{d_v}_v\,.
  \]

  Therefore, in view of \eqref{eq:productformulaetc} we get 
  \begin{align*}
    \vol_\bbA(S)&=\prod_{v\nmid\infty}\vol_v(C_v)\cdot \prod_{v\mid\infty}\vol_v(C_v)\\
    &=\prod_{v\nmid\infty} \abs{\det(A)}^{d_v}_v \cdot
    \frac{1}{(n!)^d}\prod_{v\mid\infty} \abs{\det(A)}_v^{1}\\
    &=\frac{1}{(n!)^d} \cdot 1\,.\qedhere
  \end{align*}
\end{proof}

\section{adelic triangulations}\label{sec:adelictriangulation}

Throughout  this section we assume that $K$ is totally real, i.e.,
$K_v=\bbR$ for all $v\mid\infty$, and we start with the proof of
Theorem \ref{thm:adelicblichfeldttotallyreal}
\begin{proof}[Proof of Theorem \ref{thm:adelicblichfeldttotallyreal}] Let $C\cap
  K^n=\{a_1,\ldots,a_{n+m}\}$. According to our assumption
  $\dim_K(C\cap K^n)=n$ we have $m\geq 1$, and let
  $P=\conv_\bbA\Set{a_1,\ldots,a_{n+m}}$. It suffices to prove the
  theorem for $P$. To this end fix an embedding 
$\fixedv\colon K^n\rightarrow \bbR^n$. Then 
\[
  P_{\fixedv}=\conv\Set*{\fixedv(a_1),\ldots,\fixedv(a_{n+m})}\subset\bbR^n
\]
is a polytope and 
there exists a triangulation $T_1,\ldots,T_k$ of $P_{\fixedv}$ with $k\geq m$ full-dimensional simplices, 
whose vertices are among the $\fixedv(a_i)$, see e.g.~Section 2.2 of
the book \cite{Loera:2010p9983}, to which we refer also for
more details on triangulations.
An element $T_j$ of this triangulation, i.e.\ an $(n+1)$-element set from the $n+m$ points, 
gives rise to an adelic simplex 
\[
  S_j=\conv_\bbA\Sett[\big]{a_i}{i\in T_j}
     =\conv_\bbA\Sett[\big]{a_i}{\fixedv(a_i)\text{ is a vertex of } T_j}
\]
with $S_{j,\fixedv}=T_j$.
Since the triangulation fulfills $\dim(T_j\cap T_i)<n$ for $j\neq i$,
we get $\vol_v(S_{j,\fixedv}\cap S_{i,\fixedv})=0$ and thus
$\vol_\bbA(S_j\cap S_i)=0$ for $i\neq j$.
On the other hand,
\[
  P=\conv_\bbA\Set{a_1,\ldots,a_{n+m}}\supset S_1\cup\ldots\cup S_k\,.
\]
Hence by Lemma~\ref{lem:adelicsimplexvolume} we conclude 
\begin{equation}\label{eq:adeliclatticepolytopevolumelowerbound}
  \vol_\bbA(P)\geq k\cdot \frac{1}{(n!)^d}
  \geq\frac{m}{(n!)^d}\,.
\end{equation}
\end{proof}

\begin{rem}
The construction in the proof above, however, does not give a full triangulation of $P$, since in
general 
\[
  S_1\cup\ldots\cup S_k \subsetneq P\,,
\]
see Examples~\ref{exmp:missingpint1} and \ref{exmp:missingpint2} below.
Example~\ref{exmp:missingpint2} does also show that the dependence on $m$ can not be
improved in general, whereas a minimal triangulation of $P_{\fixedv}$ does not
necessarily give rise to a minimal set of adelic simplices, as
Example~\ref{exmp:missingpint1} shows.
\end{rem}

\bigskip

\begin{exmp}\label{exmp:missingpint1}
  Let $K=\bbQ[\sqrt{2}]$ and $n=2$. Let $P$ be the adelic convex hull of
  \[
    a=(\sqrt{2},1)\,,\quad b=(1,3)\,,\quad
    c=(2,3)\quad\text{and}\quad d=(1,\sqrt{2})\in K^2\,.
  \]
  Then for $v\nmid\infty$ we get $P_v=\calO_v^2$ and
  the two convex bodies at the infinite places $v_1$ and $v_2$ 
  with corresponding real embeddings $\sigma_1$ and $\sigma_2$ are
  $P_{v_1}$ and $P_{v_2}$ as depicted in the figure.
\begin{center}
\begin{tikzpicture}[scale=.8]
\draw[gray!50,dashed,thick] (0,-1.6) -- (0,3.5);
\draw[gray!50,dashed,thick] (-1.6,0) -- (3.5,0);
\draw (1.41,1) -- (1,1.41)  -- (1,3) -- (2,3) --  cycle;
\draw[gray,thick] (1.41,1) -- (1,3);
\draw[gray,thick] (2,3) -- (1,1.41);
\draw (-1.2,3.5) node[anchor=north]  {\scriptsize$P_{v_1}$};
\draw (1.41,1) node[fill,shape=circle,scale=.3] {};
\draw (1.41,1) node[anchor=north]  {\scriptsize$\sigma_1(a)$};
\draw (2,3) node[fill,shape=circle,scale=.3] {};
\draw (2,3) node[anchor=south]  {\scriptsize$\sigma_1(b)$};
\draw (1,3) node[fill,shape=circle,scale=.3] {};
\draw (1,3) node[anchor=south]  {\scriptsize$\sigma_1(c)$};
\draw (1,1.41) node[fill,shape=circle,scale=.3] {};
\draw (1,1.41) node[anchor=east]  {\scriptsize$\sigma_1(d)$};
\draw (1.55,2.7) node[fill,shape=circle,scale=.3] {};
\draw (1.62,2.7) node[anchor=east]  {\scriptsize$z_1$};
\end{tikzpicture}
\hspace{1cm}
\begin{tikzpicture}[scale=.8]
\draw[gray!50,dashed,thick] (0,-1.6) -- (0,3.5);
\draw[gray!50,dashed,thick] (-1.6,0) -- (3.5,0);
\draw (-1.41,1) -- (1,3) -- (2,3) -- (1,-1.41)  -- cycle;
\draw[gray,thick] (-1.41,1) -- (2,3);
\draw[gray,thick] (1,3) -- (1,-1.41);
\draw (-1.2,3.5) node[anchor=north]  {\scriptsize$P_{v_2}$};
\draw (-1.41,1) node[fill,shape=circle,scale=.3] {};
\draw (-1.41,1) node[anchor=east]  {\scriptsize$\sigma_2(a)$};
\draw (2,3) node[fill,shape=circle,scale=.3] {};
\draw (2,3) node[anchor=south]  {\scriptsize$\sigma_2(b)$};
\draw (1,3) node[fill,shape=circle,scale=.3] {};
\draw (1,3) node[anchor=south]  {\scriptsize$\sigma_2(c)$};
\draw (1,-1.41) node[fill,shape=circle,scale=.3] {};
\draw (1,-1.41) node[anchor=west]  {\scriptsize$\sigma_2(d)$};
\draw (.5,1.5) node[fill,shape=circle,scale=.3] {};
\draw (.5,1.5) node[anchor=north]  {\scriptsize$z_2$};
\end{tikzpicture}
\end{center}
The adelic simplices
\begin{align*}
  S_1&=\conv_\bbA\Set{a,b,c} & S_2&=\conv_\bbA\Set{a,b,d}\\
  S_3&=\conv_\bbA\Set{a,c,d} & S_4&=\conv_\bbA\Set{b,c,d}
\end{align*}
all satisfy $S_{j,v}=\calO_v^2$ for $v\nmid\infty$
and $\vol(S_j\cap S_i)=0$ for all $j\neq i$ since the intersection at one
infinite place is always lower-dimensional.
Therefore \eqref{eq:adeliclatticepolytopevolumelowerbound} is not optimal in
this case as $P$ contains four disjoint simplices,
even though the indicated triangulations of $P_{v_1}$ and  $P_{v_2}$ are minimal.

The point $z=(z_1,z_2)$ indicated in the picture is not contained in any $S_j$.
\end{exmp}

\begin{exmp}\label{exmp:missingpint2}
  Consider again $K=\bbQ[\sqrt{2}]$ and let $P$ be the adelic convex hull of
  \[
    a=(1,1)\,,\quad b=(2,1)\,,\quad
    c=(2,2)\quad\text{and}\quad d=(1,2)\in K^2\,.
  \]
  Then for $v\nmid\infty$ we get $P_v=\calO_v^2$ and
  the two convex bodies at the infinite places $v_1$ and $v_2$ 
  with corresponding real embeddings $\sigma_1$ and $\sigma_2$ are
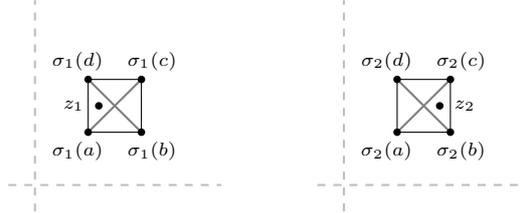
\begin{center}
\begin{tikzpicture}[scale=.7]
\draw[gray!50,dashed,thick] (0,-0.5) -- (0,3.5);
\draw[gray!50,dashed,thick] (-0.5,0) -- (3.5,0);
\draw (1,1) -- (1,2) -- (2,2) -- (2,1) -- cycle;
\draw[gray,thick] (1,1) -- (2,2);
\draw[gray,thick] (2,1) -- (1,2);
\draw (1,1) node[fill,shape=circle,scale=.3] {};
\draw (.8,1) node[anchor=north]  {\scriptsize$\sigma_1(a)$};
\draw (2,1) node[fill,shape=circle,scale=.3] {};
\draw (2.2,1) node[anchor=north]  {\scriptsize$\sigma_1(b)$};
\draw (2,2) node[fill,shape=circle,scale=.3] {};
\draw (2.2,2) node[anchor=south]  {\scriptsize$\sigma_1(c)$};
\draw (1,2) node[fill,shape=circle,scale=.3] {};
\draw (.8,2) node[anchor=south]  {\scriptsize$\sigma_1(d)$};
\draw (1.2,1.5) node[fill,shape=circle,scale=.3] {};
\draw (1.1,1.5) node[anchor=east]  {\scriptsize$z_1$};
\end{tikzpicture}
\hspace{1cm}
\begin{tikzpicture}[scale=.7]
\draw[gray!50,dashed,thick] (0,-0.5) -- (0,3.5);
\draw[gray!50,dashed,thick] (-0.5,0) -- (3.5,0);
\draw (1,1) -- (1,2) -- (2,2) -- (2,1) -- cycle;
\draw[gray,thick] (1,1) -- (2,2);
\draw[gray,thick] (2,1) -- (1,2);
\draw (1,1) node[fill,shape=circle,scale=.3] {};
\draw (.8,1) node[anchor=north]  {\scriptsize$\sigma_2(a)$};
\draw (2,1) node[fill,shape=circle,scale=.3] {};
\draw (2.2,1) node[anchor=north]  {\scriptsize$\sigma_2(b)$};
\draw (2,2) node[fill,shape=circle,scale=.3] {};
\draw (2.2,2) node[anchor=south]  {\scriptsize$\sigma_2(c)$};
\draw (1,2) node[fill,shape=circle,scale=.3] {};
\draw (.8,2) node[anchor=south]  {\scriptsize$\sigma_2(d)$};
\draw (1.8,1.5) node[fill,shape=circle,scale=.3] {};
\draw (1.9,1.5) node[anchor=west]  {\scriptsize$z_2$};
\end{tikzpicture}
\end{center}
As before, for an adelic simplex $S$ with vertices from $a,b,c,d$ it still holds
$S_v=\calO_v^2$ for $v\nmid\infty$.
Any selection of more than two simplices will contain a pair whose infinite
parts have non-trivial intersection, 
thus \eqref{eq:adeliclatticepolytopevolumelowerbound} is best possible.
Again, the point $z=(z_1,z_2)$  indicated in the picture is not contained in any of the four adelic simplices.
\end{exmp}

\section{The symmetric case}\label{sec:symmcase}

Blichfeldt's inequality has recently been improved for symmetric $C\in\calK^m_0$ by
Henze~\cite[(2.4)]{Henze:2012tr}. He proved that 
for $C\in\calK^m_0$ and $\Lambda\in\calL^m$ with
$\dim_\bbR(C\cap\Lambda)\geq m$ it holds
\begin{equation}\label{eq:henze-blichfeldt}
  \abs{C \cap \Lambda}
  \leq \frac{m!}{2^{m}}L_{m}(2)\frac{\vol_{m}(C)}{\det\Lambda}\,,
\end{equation}
where $L_{m}$ is the $m$-th Laguerre polynomial, 
$L_{m}(x)=\sum_{k=0}^{m}\binom{m}{k}\frac{x^k}{k!}$. The bound is
asymptotically sharp for certain cross-polytopes. It was also pointed
out by Henze that for any $0<\epsilon<1$ and $m=m(\epsilon)$ large enough we have 
$L_{m}(2)/2^m\leq 1/(2-\epsilon)^m$. Hence
\eqref{eq:henze-blichfeldt} is an exponential improvement on 
Blichfeldt's inequality for symmetric bodies.

In this section, we show an adelic version of Henze's inequality,
again for an arbitrary number field $K$. 
\begin{prop}\label{thm:adelichenze}
  Let $K$ be an algebraic number field of degree $d$ and
  let $C$ be a symmetric adelic convex body with $\dim_\bbQ(C \cap K^n)=nd$.
  Then
  \[
    \abs[\big]{C \cap K^n}
    \leq\frac{(nd)!}{2^{nd}}L_{nd}(2)\, \frac{\vol_\bbA(C)}{\bigl(\sqrt{\abs{\Delta_K}}\bigr)^n}\,.
  \]
\end{prop}

\begin{proof}
We use the embedding $\rho\circ\iota : K^n\rightarrow\bbR^{nd}$ and \eqref{eq:adelicandclassciallatticepointsarethesame}.
By Henze's Blichfeldt-type inequality \eqref{eq:henze-blichfeldt}, 
with
\[
  \dim_\bbQ\bigl(C \cap K^n\bigr)=\dim_\bbR\bigl(\rho(C_\infty)\cap\rho\bigl(\iota(\frakM)\bigr)\bigr)\geq nd\,,
\]
we get
\begin{align*}
  \abs[\big]{\rho\bigl(\iota(\frakM)\bigr) \cap \rho( C_\infty)}&\leq
  \frac{(nd)!}{2^{nd}}L_{nd}(2)\frac{\vol_{nd}\bigl(\rho(C_\infty)\bigr)}{\det\bigl(\rho\bigl(\iota(\frakM)\bigr)\bigr)}\\
  &=\frac{(nd)!}{2^{nd}}L_{nd}(2)\, \frac{\vol_\bbA(C)}{\bigl(\sqrt{\abs{\Delta_K}}\bigr)^n}\,.\qedhere
\end{align*}
\end{proof}


Due to the embedding argument into $\bbR^{nd}$ we have to assume that
$\dim_\bbQ(C \cap K^n)=nd$. A more adelic version, i.e., only with the
assumption $\dim_K(C \cap K^n)=n$, and with a better bound for large
degrees was proved by Gaudron,
\cite[p.~173]{Gaudron:2009iu}, using the language of heights and vector
bundles. He showed that 
for  a symmetric adelic convex body $C$ with  $\dim_K(C \cap K^n)=n$
it holds 
\begin{equation}\label{eq:gaudronsblichfeldt}
  \abs[\big]{C \cap K^n} < (5n)^{nd}\, \vol_\bbA(C)\,.
\end{equation}

For arbitrary, i.e., not necessarily $0$-symmetric, adelic convex bodies $C$ we are not aware of any results
except our Theorem \ref{thm:adelicblichfeldttotallyreal}  in the case of totally real fields and
an ``embedded version'' of Blichfeldt's inequality \eqref{eq:classicalblichfeldt}
 \[
     \abs[\big]{C\cap K^n}\leq (nd)!\,
     \frac{\vol_\bbA(C)}{\bigl(\sqrt{\abs{\Delta_K}}\bigr)^n}+nd, \]
which can be proved analogously to Proposition \ref{thm:adelichenze}. Here we
have to assume, again, $\dim_\bbQ(C \cap K^n)=nd$.

\providecommand{\bysame}{\leavevmode\hbox to3em{\hrulefill}\thinspace}
\providecommand{\MR}{\relax\ifhmode\unskip\space\fi MR }
\providecommand{\MRhref}[2]{%
  \href{http://www.ams.org/mathscinet-getitem?mr=#1}{#2}
}
\providecommand{\href}[2]{#2}

\end{document}